\documentclass[11pt,a4paper]{article}
\usepackage{amsmath}
\usepackage{amssymb}
\usepackage{amsthm}
\usepackage{amsfonts}
\usepackage{enumerate}
\usepackage{verbatim}
\usepackage{hyperref}
\usepackage{breakurl,url}
\usepackage{bm}
\usepackage{color}
\usepackage{mathtools} 
\usepackage{pstricks}
\usepackage{pst-node}
\usepackage{pst-plot}
\usepackage{algorithm}
\usepackage{algorithmic}
\usepackage{pst-poly}
\usepackage{caption,subcaption} 
\usepackage[margin=1.3in]{geometry}
\newtheorem{theorem}{Theorem}

\newtheorem{proposition}{Proposition}

\newtheorem{lemma}{Lemma}

\theoremstyle{definition}

\begin{document}

\title{A Note on the Paper ``The unique solution of the absolute value equations"}

\author{
Shubham Kumar\footnote{
Department of Mathematics, PDPM-Indian Institute of Information Technology, Design and Manufacturing Jabalpur, M.P. India, e-mail: \texttt{shub.srma@gmail.com}}
\and Deepmala\footnote{
Department of Mathematics, Faculty of Natural Sciences, PDPM-Indian Institute of Information Technology, Design and Manufacturing Jabalpur, M.P. India, e-mail: \texttt{dmrai23@gmail.com}}
}

\date{\today}
\maketitle

\begin{abstract}
In this note, we give the possible revised version of the unique solvability conditions for the two incorrect results that appeared in the published paper by Wu et al. (Appl Math Lett 76:195-200, 2018).
\end{abstract}

\bigskip

\section{Introduction}
We consider the following absolute value equations (AVE)
\begin{equation} \label{Equ AVE}
	Ax -  \vert x \vert = b,
\end{equation}
where $A \in \mathbb{R}^{n\times n},$ $b \in \mathbb{R}^{n}$ are given and $x \in \mathbb{R}^{n}$ is the unknown.



The Absolute Value Equation (AVE) was initially explored by Rohn in \cite{rohn2004theorem}. Subsequently, numerous scholars have conducted in-depth investigations; Rohn \cite{rohn2004theorem} gave the alternative theorem for the AVE and showed that AVE is equivalent to linear complementarity problems under some conditions. Mangasarian \cite{mangasarian2006absolute} gave the existence and nonexistence conditions for the AVE, and in \cite{mangasarian2007absolute} showed that solving the AVE is an NP-hard problem. Hladik \cite{hladik2023properties} provided the topological properties for the solution set of AVE, and further researchers obtained the different conditions for the unique solvability of the AVE ( see \cite{hladik2023some,wu2021unique,wu2018unique} and references therein). Wu et al. \cite{wu2018unique} showed that if the conditions $\frac{1}{2}$$|| (A+I)^{-1} ||_2 < 1$ or $\frac{1}{2}$$|| (A-I)^{-1} ||_2 < 1$ are hold then AVE (\ref{Equ AVE}) has exactly one solution, but these conditions are incorrect as showed by Hladik \cite{hladik2023some}. In this note, based on some results from the literature, we will provide the alternative to the above incorrect conditions.

\paragraph{Notation.} The maximum and minimum singular values of the matrix A are represented by $\sigma_{max}(A)$ and $\sigma_{min}(A)$, respectively.  $|| . ||_2$ denotes the spectral norm for matrices. $I$ denotes the identity matrix. 
\section{The Unique Solvability of the AVE}


\begin{theorem} \cite{wu2018unique} \label{Thm 1}
	The AVE (\ref{Equ AVE}) has exactly one solution for each b, if any of the conditions listed below are met:\\ 
	(i)  $\frac{1}{2}$$|| (A-I)^{-1} ||_2 < 1$.\\
	(ii) $\frac{1}{2}$$|| (A+I)^{-1} ||_2 < 1$.
\end{theorem}


Hladík et al. \cite{hladik2023some} show that the conditions of the Theorem \ref{Thm 1} are incorrect. By taking $A =0 \in \mathbb{R},$ all the conditions of the Theorem \ref{Thm 1} are satisfied, but AVE $-|x| = b$ does not have a unique solution for each $b \in \mathbb{R}$. 


The possible revised versions of the Theorem \ref{Thm 1} are given in the following results.
\begin{proposition} \label{Thm 3}
	If $2$$|| (A+I)^{-1} ||_{2} < 1$  then the AVE (\ref{Equ AVE}) has exactly one solution for each b.
\end{proposition}
\begin{proof}
	If $\sigma_{min}(A+I) > 2$ then the AVE (\ref{Equ AVE}) has exactly one solution \cite{wu2021unique}.
	By relation, $\sigma_{min}(X).\sigma_{max}(X^{-1}) = 1,$ condition $\sigma_{min}(A+I) > 2$ is equivalent to the condition $\sigma_{max}(A+I)^{-1} < \frac{1}{2}.$ This completes the proof.	
\end{proof}

\begin{lemma} \cite{zhang2011matrix} \label{Lemma 1}
	The inequality $\sigma_{min}(A) - \sigma_{max}(B)$ 	$\leq$ $\sigma_{min}(A+B)$ holds for the matrices  A, B $\in \mathbb{R}^{n \times n}$.
\end{lemma}

\begin{proposition}  \label{Thm 4}
	If $\sigma_{min}(A-I)>2$ then the AVE (\ref{Equ AVE}) has exactly one solution for each b.
\end{proposition}
\begin{proof}
	If $A-I+2D$ is invertible for each $D \in [0, I]$ then the AVE (\ref{Equ AVE}) has exactly one solution \cite{wu2018unique}. Based on the Lemma \ref{Lemma 1}, we have $\sigma_{min}(A-I+2D) \geq \sigma_{min}(A-I) - \sigma_{max}(2D).$ Since $\sigma_{max}(2D) \leq 2,$ so if $\sigma_{min}(A-I) > 2$ then $\sigma_{min}(A-I+2D) > 0.$ This shows that matrix $A-I+2D$ is invertible. This completes the proof.
\end{proof}

\begin{proposition}  \label{Thm 5}
	If $2$$|| (A-I)^{-1} ||_{2} < 1$  then the AVE (\ref{Equ AVE}) has exactly one solution for each b.
\end{proposition}
\begin{proof}	
	By relation, $\sigma_{min}(X).\sigma_{max}(X^{-1}) = 1,$  and Theorem \ref{Thm 4}, the condition $\sigma_{min}(A-I)>2$ can be written as $\sigma_{max}(A-I)^{-1} < \frac{1}{2}.$ This completes the proof.
\end{proof}


\paragraph{Acknowledgments.} 
The research work of Shubham Kumar was supported by the Ministry of Education, Government of India, through Graduate Aptitude Test in Engineering (GATE) fellowship registration No. MA19S43033021.



\bibliographystyle{abbrv}
\bibliography{ave_short_note}

\end{document}